\documentclass{amsart}

\usepackage{amssymb,tikz}
\usepackage{amsthm}
\usepackage{amsmath, a4wide}
\usepackage{amsbsy}
\usepackage[all]{xy}
\usepackage{bm}
\usepackage{hyperref}
\usepackage{color}

\begin{document}

\newtheorem{thm}{Theorem}[section]
\newtheorem{cor}[thm]{Corollary}
\newtheorem{conj}[thm]{Conjecture}
\newtheorem{lem}[thm]{Lemma}
\newtheorem{defi}[thm]{Definition}
\newtheorem{prop}[thm]{Proposition}
\newtheorem{opg}[thm]{Opgave}
\newtheorem{rem}[thm]{Remark}
\newtheorem{ex}[thm]{Example}
\newtheorem{cond}[thm]{Condition}
\newtheorem*{pro}{Problem}

\newenvironment{Proof}{\begin{trivlist}

\item[\hskip\labelsep{\it Proof.}]}{$\hfill\Box$\end{trivlist}}

\newcommand{\NN}{\mathbb{N}}
\newcommand{\ZZ}{\mathbb{Z}}
\newcommand{\RR}{\mathbb{R}}
\newcommand{\EE}{\mathbb{E}}
\newcommand{\PP}{\mathbb{P}}
\newcommand{\Ss}{\mathbb{S}}
\newcommand{\MM}{\mathbb{M}}

\newcommand{\BB}{\mathbb{B}}

\newcommand{\Pset}{\mathcal{P}}
\newcommand{\Gset}{\mathcal{G}}
\newcommand{\Sset}{\mathcal{S}}

\title{On the Discrepancy of Jittered Sampling}

\author{Florian Pausinger}
\address{IST Austria, Am Campus 1, 3400 Klosterneuburg, Austria}

\author{Stefan Steinerberger}
\address{Department of Mathematics, Yale University, 10 Hillhouse Avenue, New Haven, CT 06511, USA}
\begin{abstract} We study the discrepancy of jittered sampling sets: such a set $\mathcal{P} \subset [0,1]^d$ is generated for fixed $m \in \mathbb{N}$ by partitioning
$[0,1]^d$ into $m^d$ axis aligned cubes of equal measure and placing a random point inside each of the $N = m^d$ cubes. We prove that,
for $N$ sufficiently large,
 $$ \frac{1}{10}\frac{d}{N^{\frac{1}{2} + \frac{1}{2d}}} \leq \mathbb{E} D_N^*(\mathcal{P})  \leq  \frac{\sqrt{d} (\log{N})^{\frac{1}{2}}}{N^{\frac{1}{2} + \frac{1}{2d}}},$$
where the upper bound with an unspecified constant $C_d$ was proven earlier by Beck.
Our proof makes crucial use of the sharp Dvoretzky-Kiefer-Wolfowitz inequality and a suitably taylored Bernstein inequality; we have reasons to believe
that the upper bound has the sharp scaling in $N$. Additional heuristics suggest that jittered sampling should be able to improve
 known bounds on the inverse of the star-discrepancy in the regime $N \gtrsim d^d$. We also prove a \textit{partition principle} showing that every partition of $[0,1]^d$ combined with a jittered sampling construction gives rise to a set whose expected squared $L^2-$discrepancy is smaller than that of purely random points.
\end{abstract}
\maketitle


\section{Introduction}

\subsection{Introduction.} Given a set $X = \left\{x_1, \dots, x_N\right\}$ of $N$ points in $[0,1]^d$ we define, as usual, the star discrepancy
$$D_N^* (X) = \sup_{R \subset[0,1]^d}{ \left| \frac{ \# \left\{i: x_i \in R \right\}}{N} - |R| \right|},$$
where the supremum ranges over all rectangles having all sides parallel to the axes and anchored in the origin. It is easy to see that a regular grid of $N$ points has discrepancy $D_N^* \sim N^{-1/d}$. In contrast, if we take $X$ to be a
collection of $N$ independently and uniformly distributed random variables, then it is known (see \cite{dorr,hnww}) that $D_N^* \sim N^{-1/2}$ with
some nonzero probability.
Jittered sampling combines the best of both worlds by partitioning $[0,1]^d$ into $m^d$ axis aligned cubes of equal measure and placing a random point inside each of the $N = m^d$ cubes; see Figure \ref{fig:def}.

\begin{center}
\begin{figure}[h!]
\centering
\begin{tikzpicture}[scale=0.65]
\draw[step=1cm,gray,very thin] (0,0) grid (5,5);
\draw[step=1cm,gray,very thin] (8,0) grid (13,5);
\node at (0,0) {$\bullet$}; 
\node at (0,1) {$\bullet$}; 
\node at (0,2) {$\bullet$}; 
\node at (0,3) {$\bullet$}; 
\node at (0,4) {$\bullet$}; 

\node at (1,0) {$\bullet$}; 
\node at (1,1) {$\bullet$}; 
\node at (1,2) {$\bullet$}; 
\node at (1,3) {$\bullet$}; 
\node at (1,4) {$\bullet$}; 

\node at (2,0) {$\bullet$}; 
\node at (2,1) {$\bullet$}; 
\node at (2,2) {$\bullet$}; 
\node at (2,3) {$\bullet$}; 
\node at (2,4) {$\bullet$}; 

\node at (3,0) {$\bullet$}; 
\node at (3,1) {$\bullet$}; 
\node at (3,2) {$\bullet$}; 
\node at (3,3) {$\bullet$}; 
\node at (3,4) {$\bullet$}; 

\node at (4,0) {$\bullet$}; 
\node at (4,1) {$\bullet$}; 
\node at (4,2) {$\bullet$}; 
\node at (4,3) {$\bullet$}; 
\node at (4,4) {$\bullet$}; 

\node at (8.31,0.23) {$\bullet$}; 
\node at (8.7,1.5) {$\bullet$}; 
\node at (8.2,2.7) {$\bullet$}; 
\node at (8.3,3.34) {$\bullet$}; 
\node at (8.9,4.29) {$\bullet$}; 

\node at (9.3,0.74) {$\bullet$}; 
\node at (9.55,1.56) {$\bullet$}; 
\node at (9.1,2.83) {$\bullet$}; 
\node at (9.65,3.42) {$\bullet$}; 
\node at (9.45,4.43) {$\bullet$}; 

\node at (10.34,0.41) {$\bullet$}; 
\node at (10.67,1.66) {$\bullet$}; 
\node at (10.52,2.42) {$\bullet$}; 
\node at (10.67,3.37) {$\bullet$}; 
\node at (10.52,4.12) {$\bullet$}; 

\node at (11.5,0.82) {$\bullet$}; 
\node at (11.7,1.58) {$\bullet$}; 
\node at (11.3,2.21) {$\bullet$}; 
\node at (11.9,3.14) {$\bullet$}; 
\node at (11.7,4.72) {$\bullet$}; 

\node at (12.23,0.63) {$\bullet$}; 
\node at (12.56,1.61) {$\bullet$}; 
\node at (12.89,2.25) {$\bullet$}; 
\node at (12.17,3.57) {$\bullet$}; 
\node at (12.67,4.19) {$\bullet$}; 
\end{tikzpicture}
\caption{The regular grid and a point set obtained by jittered sampling.} \label{fig:def}
\end{figure}
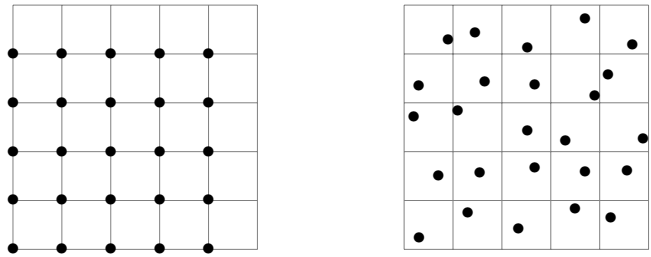
\end{center}
\vspace{-15pt}
This idea seems to have first been explored in 1981 by Bellhouse \cite{bellhouse}, where he shows 'that a stratified sampling design,
although less convenient to implement than systematic sampling, is usually more efficient.'  The point sets make a reappearance 
in computer graphics in a 1984 paper of Cook, Porter \& Carpenter \cite{port} by the name of \textit{jittered sampling}. The bound
$$\mathbb{E} D_N^*(\mathcal{P})  \leq C_d \frac{(\log{N})^{\frac{1}{2}}}{N^{\frac{1}{2} + \frac{1}{2d}}}$$
was proven by Beck \cite{beck87} (see also the book Beck \& Chen \cite{beckchen} or the exposition in Chazelle \cite{chazelle}).
Beck actually derives the result for the more general notion of discrepancy w.r.t. to rotations of a convex set; the proof does not give any
information about the constant or how it grows as a function of the dimension. 
There exist deterministic lower bounds valid for
all sets of points (see Chen \& Travaglini \cite{chen}).
We are exclusively concerned with the behavior of random points.

\subsection{Main result.} We restrict ourselves to the notion of star-discrepancy, where our
contribution is a proof yielding explicit control on the constant as a function of dimension. We hope that our paper will spark interest in jittered sampling
as a possible avenue towards improved bounds for the inverse of the star-discrepancy.

\begin{thm} \label{main} For a random set $\mathcal{P} \subset [0,1]^d$ with $N = m^d$ points obtained from jittered sampling, we have, for $N$ sufficiently large depending on $d$,
$$ \frac{1}{10} \frac{d}{N^{\frac{1}{2} + \frac{1}{2d}}} \leq \mathbb{E} D_N^*(\mathcal{P})  \leq  \frac{\sqrt{d} (\log{N})^{\frac{1}{2}}}{N^{\frac{1}{2} + \frac{1}{2d}}}.$$
\end{thm}
For $d=1$ it is easy to see that $D_N^*(\mathcal{P}) \sim N^{-1}$. However, we have reason to believe that for $d \geq 2$ the upper bound could have the correct order as $N \rightarrow \infty$ (up to a
constant depending on $d$). These
considerations are detailed at the end of the paper where we comment on a possible approach to improving the lower bound and carry out corresponding
back-of-the-envelope estimates. 
We believe our main contribution is  
\begin{enumerate}
\item to show that jittered sampling is amenable to being analyzed at a very fine level;
\item to show that jittered sampling might be a natural way to establish improved bounds on the inverse star discrepancy (previously suggested by Aistleitner \cite{aisti}); and, in that spirit,
\item to ask whether the upper bound really requires $N$ to be large (we believe not).
\end{enumerate}
It seems natural to expect $\mathbb{E} D_N^*(\mathcal{P})$ to \textit{always} (not just for $N$ sufficiently large) be smaller than the average star discrepancy of the same number of purely random points (some numerical 
experiments for $N$ small can be found in Section 6). We have been unable to prove that for $D_N^*$ but were able to show it at a much greater level of generality for the $\mathcal{L}^2-$discrepancy 
(see below).

\subsection{Relationship with prior results.} 
Let us compare our result with the celebrated result of Heinrich, Novak, Wasilkowski \& Wozniakowski \cite{hnww} who showed the existence of 
a set of $N$ points in $[0,1]^d$ with
$$ D_N^*(\mathcal{P}) \leq c \sqrt{\frac{d}{N}} \qquad \mbox{for some universal constant}~c.$$
Aistleitner \cite{aisti}, using a result of Gnewuch \cite{gnew}, has shown that one can take $c=10$ and Doerr \cite{dorr} has shown this to be the correct order of magnitude for a set $\mathcal{P}$ of $N$ points chosen independently and uniformly at random from $[0,1]^d$, i.e.
$\mathbb{E}  D_N^*(\mathcal{P}) \gtrsim \sqrt{d/N}$. 
Furthermore, we refer to papers of Gnewuch \cite{gne} and Hinrichs \cite{hin} (see also \cite{sst}) as well as the recent comprehensive monographs of
Novak \& Wozniakowski \cite{nov1, nov2, nov3}. If our upper bound were to hold unconditionally for all $N$ (possibly with a different absolute constant $c$ in front), then this would improve the known estimate (see \cite{hnww}) 
$$ \mbox{disc}^{*}(N,d) \leq c \sqrt{\frac{d}{N}} \qquad \mbox{to} \qquad  \mbox{disc}^{*}(N,d) \leq c \sqrt{\frac{d}{N} \min\left\{1, \frac{\log{N}}{N^{\frac{1}{d}}} \right\}   },$$
where 
$$\mbox{disc}^{*}(N,d):= \underset{\mathcal{P}}\inf \ D_N^*(\mathcal{P}).$$
This new bound improves on the old one if $N$ is slightly bigger than $d^d$ ('slightly bigger' being understood on a logarithmic scale). Even though $N \sim d^d$ is fairly large, we do not know of any better constructions in that regime (Section 6 contains a heuristic
comparison with upper bounds on the discrepancy of Hammersley point sets).
There exists a natural reason why one would expect jittered sampling to
suddenly gain in effectiveness around $N \sim d^d$, this is also detailed in Section 6. Section 6 furthermore contains a heuristic argument suggesting
$$ \mathbb{E} D_N^*(\mathcal{P})  \gtrsim  \frac{d + (\log{N})^{\frac{1}{2}}}{N^{\frac{1}{2} + \frac{1}{2d}}}.$$

\subsection{A general partition principle.} The advantage of jittered sampling is that, by construction, the point set avoids a certain type of clustering. The purpose of this section is to show that the same underlying principle might hold at a greater level of generality and to establish it for the $\mathcal{L}_2-$discrepancy in particular. Formally, consider a partition of the unit cube into $N$ Lebesgue-measurable sets of equal measure
$$ [0,1]^d = \bigcup_{i=1}^{N}{\Omega_i} \qquad \mbox{such that} \quad \forall~1 \leq i \leq N: |\Omega_i| = \frac{1}{N}.$$
We pose no other geometric conditions on the partition, in particular the $\Omega_i$ need not be connected. 
Any such partition gives rise to a $N-$element `jittered-sampling' set $\mathcal{P}_{\Omega}$ adapted to the partition by picking a random point (random w.r.t. the Lebesgue measure) from each $\Omega_i$.
\begin{center}
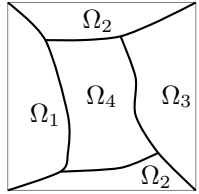
\begin{figure}[h!]
\centering
\begin{tikzpicture}[scale=0.5]
\draw[gray,very thin] (0,0) -- (5,0);
\draw[gray,very thin] (5,5) -- (5,0);
\draw[gray,very thin] (5,5) -- (0,5);
\draw[gray,very thin] (0,0) -- (0,5);
\draw[thick] plot [smooth,tension=0.5] coordinates{(0,0) (1.4,0.5) (1.6, 1) (1.6, 2) (1,4) (0,5) };
\node (A) at (1,2)  {$\Omega_1$};
\node (A) at (2.4,4.5)  {$\Omega_2$};
\draw[thick] plot [smooth,tension=0.5] coordinates{(1,4) (2,4) (3, 4.1) (4, 4.5) (5, 5) };
\draw[thick] plot [smooth,tension=0.5] coordinates{(3, 4.1) (3.4, 3) (3.4, 2) (4,1) (5,0) };
\node (A) at (4.5,2.5)  {$\Omega_3$};
\draw[thick] plot [smooth,tension=0.5] coordinates{ (1.4,0.5) (3,0.6) (4,1) };
\node (A) at (3.95,0.35)  {$\Omega_2$};
\node (A) at (2.5,2.5)  {$\Omega_4$};
\end{tikzpicture}
\caption{A partition of $[0,1]^2$ into 4 sets of equal measure ($\Omega_2$ is disconnected).}
\end{figure}
\end{center}
\vspace{-20pt}
We believe that such point sets $\mathcal{P}_{\Omega}$ are on average more regular than sets comprised of $N$ purely random points $\mathcal{P}_N$ \textit{independently of the partition $\Omega$} in a very general sense. We prove it in one special case: recall that the $\mathcal{L}^2-$disrepancy of a set of points $A \subset [0,1]^d$ is defined as
the $L^2-$norm of the discrepancy function and can be written as
$$\mathcal{L}_2(A) := \left( \int_{[0,1]^d}{ \left| \frac{ \#A \cap [0, \textbf{x}]}{\#A} -  \left|  [0,\textbf{x}]\right|   \right|^2 d\textbf{x}}\right)^{\frac{1}{2}}.$$
\begin{thm}[Partition principle] Given any partition of $[0,1]^d$ into $N$ sets of equal measure
$$ [0,1]^d = \bigcup_{i=1}^{N}{\Omega_i} \qquad \mbox{such that} \quad \forall~1 \leq i \leq N: |\Omega_i| = \frac{1}{N},$$
the jittered sampling construction $\mathcal{P}_{\Omega}$ associated to $\Omega$ satisfies
$$ \mathbb{E}~ \mathcal{L}_2(\mathcal{P}_{\Omega})^2 \leq  \mathbb{E} ~\mathcal{L}_2(\mathcal{P}_N)^2.$$ 
\end{thm}
The result is sharp: for every $N \in \mathbb{N}$ it is possible to construct a sequence of partitions $\Omega_1, \Omega_2, \dots$ such that
$$ \lim_{k \rightarrow \infty}{\mathbb{E}~ \mathcal{L}_2(\mathcal{P}_{\Omega_k})^2} =  \mathbb{E} ~\mathcal{L}_2(\mathcal{P}_N)^2.$$ 
This can be done by taking finer and finer grids and assigning the little cubes randomly to one of the $N$ partitions. The law of large numbers guarantees that each partition element gets
a proportion of $(1+o(1))/N$ of all possible cubes; a simple computation will then yield the result. This is not at all surprising since that construction emulates the purely random points $\mathcal{P}_{N}$ 
and discrepancy is continuous.
\subsection{Open questions.}
\textbf{Jittered sampling.} The most important question seems to be whether it is possible to precisely determine the asymptotic behavior of the discrepancy of jittered sampling. Our proof 
gives a slightly stronger result and allows us to conclude that 
$$ \mbox{for every}~ C > \sqrt{\frac{3}{4} + \frac{1}{4d}} \quad \mbox{we have} \quad \mathbb{P}\left( D_N^*(\mathcal{P}) \geq   C\frac{\sqrt{d} (\log{N})^{\frac{1}{2}}}{N^{\frac{1}{2} + \frac{1}{2d}}}\right) \leq \frac{1}{N^{C^2}} $$
for $N$ sufficiently large (depending on $d,C$).
However, this merely improves the constant in the main result, not the asymptotic order. Is it true that
$$ \mathbb{E} D_N^*(\mathcal{P})  \gtrsim  \frac{d + (\log{N})^{\frac{1}{2}}}{N^{\frac{1}{2} + \frac{1}{2d}}} \qquad \mbox{or maybe even} \qquad  \mathbb{E} D_N^*(\mathcal{P})  \sim  \frac{d + (\log{N})^{\frac{1}{2}}}{N^{\frac{1}{2} + \frac{1}{2d}}}?$$

\textbf{Partition principle.} We believe the partition principle to have some importance: in particular, if one tries to do numerical integration with a fixed number $N$ of points, then using jittered sampling with respect to \textit{any} partition into $N$ sets of equal measure is guaranteed to give a smaller value for $\mathcal{L}_2(\mathcal{P}_{\Omega})^2$ on average. In contrast to classical jittered sampling using a grid (where we have to require $N = m^d$) there is no restriction on $N$ nor is there any restriction on the partition. 

\begin{quote}
\textbf{Question.} Is there an analogue of Theorem 1.2. for other (say, $\mathcal{L}_p$) discrepancies?
\end{quote}

It is clear that not all partitions will be equally effective and it is also tempting to believe that jittered sampling might just be the best way to minimize $D_N^*$; we actually believe that this could perhaps not be the case: maybe it's advantageous to break the symmetry and make the cubes to the origin a little bit bigger (which introduces a 'systematic' error close to the origin but a small one); the advantage being that (slightly) larger cubes close to the origin eat up more space and allow for the remaining cubes to be slightly smaller and thus reduce the error in $[0, \textbf{x}]$ for $\textbf{x}$ close to $(1,1, \dots , 1)$. This motivates the following question

\begin{quote}
\textbf{Question.} Which partition minimizes the $D_N^*$-discrepancy on average? Is it possible to show that jittered sampling is not optimal?
\end{quote}

 At least the second part of the question could perhaps be addressed by an explicit numerical investigation (for example by considering an explicit partition of $[0,1]^2$ into $N=36$ parts and comparing it with jittered sampling). If we were ask the same question for the $D_N-$discrepancy, i.e. the quantity
$$D_N (X) = \sup_{R \subset[0,1]^d}{ \left| \frac{ \# \left\{i: x_i \in R \right\}}{N} - |R| \right|},$$
where $R$ now ranges over \textit{all} axis-parallel rectangles (not just the ones anchored in the origin), then it would seem conceivable that jittered sampling could give the best results.

\subsection{Organization of the paper.} We start by giving a proof of Theorem \ref{main} in two dimensions in Section 2. The argument introduces the Dvoretzky-Kiefer-Wolfowitz inequality
as well as the crucial geometric decomposition and should be helpful in understanding the higher-dimensional argument. Section 3 gives a precise definition of the geometric decomposition and contains our adapted Bernstein inequality. Section 4  and Section 5 give a proof of the upper and lower bound, respectively, and Section 6 contains additional remarks and comments. Section 7 gives a proof of the partition principle.


\section{The proof in two dimensions}

We start by giving a simplified proof of the two-dimensional case. This case is substantially easier than the higher-dimensional case 
but it conveys the main structure of the argument. We will not yet employ the adapted Bernstein inequality, hence the constant in the result is actually larger (2 instead of $\sqrt{2}$) than what is achieved
by the main result.

\begin{proof} We denote the discrepancy function by
$$ f(x,y) := \left| \frac{\# \left\{ p \in \mathcal{P} \cap [0,x] \times [0,y]\right\}}{|\mathcal{P}|} - xy\right|.$$ Following the definition of our point sets, we see that computing the discrepancy function reduces to understanding two strips, see Figure \ref{fig:strips}, because
for all cubes fully contained inside $[0,\textbf{x}]$, for $\textbf{x}=(x,y)$, the contribution to the discrepancy is 0 by construction. The idea
of arranging points in such a way that the computation of discrepancy reduces to sets other than full hyperrectangles
is certainly not new and started the classical theory of $(t,s)-$sequences (see Niederreiter \cite{ni}).
\begin{center}
\begin{figure}[h!]
\centering
\begin{tikzpicture}[scale=1]
\fill[gray!20!white] (12.67,0) -- (12.67,4) -- (12,4)--(12,0);
\fill[gray!20!white] (8,4) -- (12,4) -- (12,4.19)--(8,4.19);
\draw[step=1cm,gray,very thin] (8,0) grid (13,5);
\draw[dashed,thick] (12.67,0) -- (12.67,4.19) -- (8,4.19);
\node at (8.31,0.23) {$\bullet$}; 
\node at (8.7,1.5) {$\bullet$}; 
\node at (8.2,2.7) {$\bullet$}; 
\node at (8.3,3.34) {$\bullet$}; 
\node at (8.9,4.29) {$\bullet$}; 
\node at (9.3,0.74) {$\bullet$}; 
\node at (9.55,1.56) {$\bullet$}; 
\node at (9.1,2.83) {$\bullet$}; 
\node at (9.65,3.42) {$\bullet$}; 
\node at (9.45,4.43) {$\bullet$}; 
\node at (10.34,0.41) {$\bullet$}; 
\node at (10.67,1.66) {$\bullet$}; 
\node at (10.52,2.42) {$\bullet$}; 
\node at (10.67,3.37) {$\bullet$}; 
\node at (10.52,4.12) {$\bullet$}; 
\node at (11.5,0.82) {$\bullet$}; 
\node at (11.7,1.58) {$\bullet$}; 
\node at (11.3,2.21) {$\bullet$}; 
\node at (11.9,3.14) {$\bullet$}; 
\node at (11.7,4.72) {$\bullet$}; 
\node at (12.23,0.63) {$\bullet$}; 
\node at (12.56,1.61) {$\bullet$}; 
\node at (12.89,2.25) {$\bullet$}; 
\node at (12.17,3.57) {$\bullet$}; 
\node at (12.43,4.89) {$\bullet$}; 
\end{tikzpicture}
\caption{Strips of interest for a particular point $(x,y)$.} \label{fig:strips}
\end{figure}
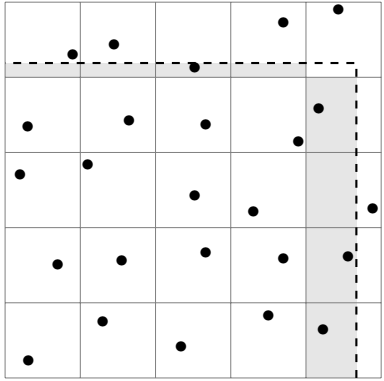
\end{center}

We are interested in the pair $(x,y)$ yielding the largest absolute value, i.e. the discrepancy of the point set. This can be rewritten as
$$ \label{discF}
\sup_{0 \leq x,y \leq 1} | f(x,y) | = \max_{1 \leq i,j \leq m }{ \sup_{\frac{i-1}{m} \leq x \leq \frac{i}{m} \atop
\frac{j-1}{m} \leq y \leq \frac{j}{m}}}{| f(x,y) | }.
$$
Consider $i,j$ to be fixed from now on. 
As Figure \ref{fig:strips} shows, it suffices to compute discrepancy arising from the two strips and the strips are almost independent from each other (and become fully independent after removing one point). We introduce two local discrepancy functions (associated to the two strips). Note that $x,y$ are again numbers between $0$ and $1$ now parametrizing the short side of the strip.
\begin{align*}
 f_1(i,j,y)=f_1(y) &= \left| \frac{\# \left\{1\leq k \leq N: 0 \leq x_k \leq \frac{i-1}{m} \wedge \frac{j-1}{m} \leq y_k \leq  \frac{j-1 + y}{m} \right\}}{|\mathcal{P}|} - \frac{i-1}{m^2}y\right|, \\
 f_2(i,j,x)=f_2(x) &= \left| \frac{\# \left\{1\leq k \leq N: \frac{i-1}{m} \leq x_k \leq \frac{i-1+x}{m} \wedge 0 \leq y_k \leq  \frac{j-1}{m} \right\}}{|\mathcal{P}|} - \frac{j-1}{m^2}x\right|.
\end{align*}
First of all, we note that by construction there is precisely one point contained in 
$[\frac{i-1}{m}, \frac{i}{m}] \times [\frac{j-1}{m} , \frac{j}{m}]$,
such that
$$ \left| \sup_{\frac{i-1}{m} \leq x \leq \frac{i}{m} \atop
\frac{j-1}{m} \leq y \leq \frac{j}{m}} | f(x,y) |   - \sup_{0 \leq y \leq 1} | f_1(y) |   - \sup_{0 \leq x \leq 1} | f_2(x) | \right| \leq \frac{1}{N}.$$

For fixed $i,j$ the functions $f_1, f_2$ measure the maximal deviation from the (uniform) limiting distribution (or, put differently, the discrepancy of the
projection of the point sets within each strip). 
Recall the \emph{Dvoretzky-Kiefer-Wolfowitz
inequality} \cite{dv} (with the sharp constant due to Massart \cite{mam}): if $z_1, z_2, \dots, z_k$ are independently and uniformly distributed random variables in $[0,1]$, then
$$ \mathbb{P}\left( \sup_{0 \leq z \leq 1}{\left| \frac{\# \left\{1 \leq \ell \leq k: 0 \leq z_{\ell} \leq z\right\}}{k} -z \right|} > \varepsilon \right) \leq 2e^{-2k\varepsilon^2}.$$

The result is purely one-dimensional and sharp (the tail bound corresponds to the first term of the Kolmogorov distribution).
We use this inequality to bound $f_1$ and $f_2$. Note that the computation of $f_1$ corresponds to $k = (i-1)$ and then rescaling the domain by a factor of $(i-1)/m^2$. 
Therefore
$$ \mathbb{P}\left( \sup_{0 \leq y \leq 1} | f_1(y) | >  \frac{i-1}{m^2}\varepsilon\right) \leq 2e^{-2(i-1)\varepsilon^2}$$
or, by setting $\delta = ((i-1)/m^2) \varepsilon$ and using $i \leq m$
$$ \mathbb{P}\left( \sup_{0 \leq y \leq 1} | f_1(y) | >  \delta \right) \leq 2e^{-2(i-1)\frac{\delta^2 m^4}{(i-1)^2}} \leq  2e^{-2\delta^2 m^3}.$$
Setting 
$$ \delta = 2\frac{(\log{N})^{1/2}}{N^{\frac{3}{4}}},$$
we get that 
$$ \mathbb{P}\left( \sup_{0 \leq y \leq 1} | f_1(y) | >  \frac{\delta}{2} \right) \leq \frac{2}{N^{2}}.$$
The same inequality holds, by symmetry, for $f_2$. Suppose now that all $f_1, f_2$ are bounded from above by
$\delta/2$ for all $i,j$. Then, by construction,
$$ D_N(\mathcal{P}) \leq \delta + \frac{1}{N}.$$
There are $2N$ strips: using the union bound, we see that the probability of one of them being bigger than $\delta/2$ is bounded from above by
$4N^{-1}$ and therefore
$$ \mathbb{E} D_N(\mathcal{P}) \leq \frac{2(\log{N})^{1/2}}{N^{\frac{3}{4}}} + \frac{4}{N^{}} + \frac{1}{N}.$$
\end{proof}


\section{Tools for the general case}

The proof of the multi-dimensional case follows the same ideas as the proof of the two-dimensional case, where we were able to decompose things into two essentially independent
strips. The main complication is that the $d-$dimensional case requires a decomposition into $2^d$ cases, all of which need to be accounted for (and $d$ of those continue to play
the role of strips and yield the main contribution). We always assume that $N = m^d$ for some $m \in \mathbb{N}$ and keep $m$ fixed throughout the argument. For ease of notation, we now define 
$\left[ \cdot \right]:\mathbb{R} \rightarrow \mathbb{R}$ to be
$$ \left[x \right] := \left\lfloor mx  + 1 \right\rfloor.$$
Furthermore, applied to a vector $\textbf{x} \in [0,1]^d$, we want $\left[\textbf{x}\right]$ to act component-wise. $\left[\textbf{x}\right]$ should be thought of as giving the coordinates of the tiny
cube with side length $m^{-1} = N^{-1/d}$ to which $\textbf{x}$ is associated; see Figure \ref{fig:cubes}. Note that 
$$\left[ \cdot \right] : [0,1]^d \rightarrow \left \{ 1, \dots, m\right \}^d$$ 
except for a set of measure 0 (the boundary), which we ignore.

\begin{center}
\begin{figure}[h!]
\centering
\begin{tikzpicture}[scale=1]
\draw[step=1cm,gray,very thin] (8,0) grid (11,3);
\node at (8.5,0.5) {$(1,1)$}; 
\node at (8.5,1.5) {$(1,2)$}; 
\node at (8.5,2.5) {$(1,3)$}; 
\node at (9.5,0.5) {$(2,1)$}; 
\node at (9.5,1.5) {$(2,2)$}; 
\node at (9.5,2.5) {$(2,3)$}; 
\node at (10.5,0.5) {$(3,1)$}; 
\node at (10.5,1.5) {$(3,2)$}; 
\node at (10.5,2.5) {$(3,3)$}; 
\end{tikzpicture} 
\caption{$\left[ \cdot \right]$ enumerating entries of matrices.} 
\label{fig:cubes}
\end{figure}
\end{center}
\subsection{Decomposition. }  Let $\mathcal{P}$ be a point set obtained by jittered sampling. We have to find a way of analyzing

$$ \sup_{\left[\textbf{x}\right] = \textbf{k}}{\left| \frac{\# \left\{\textbf{p} \in \mathcal{P}: \textbf{p} \leq \textbf{x} \right\}}{|\mathcal{P}|} - \prod_{i=1}^{d}x_i\right|},$$
where $\textbf{k} \in  \left\{1, \dots, m\right\}^d$ is a fixed string of coordinates fixing one of the small cubes, $\leq$ is to be understood component-wise and
$\textbf{x} = (x_1, \dots, x_d).$ Consider now $\textbf{k}$ fixed: we ignore the one point $\textbf{p} \in \mathcal{P}$ for which $\left[\textbf{p}\right] = \textbf{k}$ (thereby
increasing the discrepancy by at most $N^{-1}$). In the two-dimensional case, the problem decoupled into 4 different elements: the big rectangle contributing nothing
to discrepancy, the one point in small square itself that could be ignored and the two strips. In higher
dimensions, the situation is not quite as simple but a similar argument can be applied: for every $\textbf{x}$, we construct a decomposition of
$ \left\{\textbf{p} \in \mathcal{P}: \textbf{p} \leq \textbf{x} \right\}$ into $2^d$ sets
$$ \left\{\textbf{p} \in \mathcal{P}: \textbf{p} \leq \textbf{x} \right\} = \bigcup_{j=1}^{2^d}{  \left\{\textbf{p} \in \mathcal{P}: \textbf{p} \in A_j(\textbf{x}) \right\}},$$
where the $A_j(\textbf{x})$ are a partition of the hyperrectangle $[0, \textbf{x}]$.
We describe the sets $A_j(\textbf{x})$ explicitly and identify $j$ with an element from $\left\{0, 1 \right\}^d$ (i.e. a string of 0/1 of length $d$):
we have $\textbf{z} \in A_j(\textbf{x})$ if $\textbf{z} \leq \textbf{x}$ and if additionally $\left[z_i \right] = \left[x_i \right]$ for all $1 \leq i \leq d$ where the string $j$ has a 1 at position $i$ and
$\left[z_i \right] < \left[x_i \right]$ whenever the string $j$ has a 0 at that position. Put differently, for a string $L \in \left\{0,1\right\}^d$,
$$A_L(\textbf{x}) := \left\{\textbf{z} \in [0,\textbf{x}]: \forall \ 1 \leq i \leq d: \left[\textbf{z}_i\right] =  \left[\textbf{x}_i\right] \ \ \text{iff} \ \ L_i=1. \right\}.$$
It is clear that this induces a partition of $\left\{ \textbf{z} : \textbf{z} \leq \textbf{x} \right\}$: for every $\textbf{z}$ we can simply check for each coordinate whether $\left[\textbf{z}_i\right] = \left[\textbf{x}_i\right]$
and write down a 1 if this is the case or a 0 if not, thus obtaining the corresponding vector $L$. 
We use $|L|$ to denote the sum of all components and state some simple properties of this decomposition. An elementary counting argument shows
$$ \#(\mathcal{P} \cap A_L(\textbf{x}) ) \leq m^{d-|L|}.$$
This follows immediately from the more precise 
bound
$$  \#(\mathcal{P} \cap A_L(\textbf{x})) \leq \prod_{j=1}^{d}{ \begin{cases} \textbf[x_i] \qquad &\mbox{if} \quad L_j = 0 \\ 1 \qquad &\mbox{otherwise.} \end{cases}}$$
The biggest set is therefore $L = (0,0,\dots,0)$, which corresponds to the union of all cubes fully contained inside $[0, \textbf{x}]$ (which, by construction, has discrepancy 0). 
As we will see in the proof, the main contribution comes from the $d$ sets indexed by $(1, 0, \dots, 0),(0, 1, 0 \dots, 0), \dots ,(0,0, \dots, 0,1)$.

\subsection{A projection agument.} Let now $\textbf{k}$ be fixed and assume furthermore that $L$ is fixed. The purpose of this section is to point out that the computation of
$$\sup_{\left[\textbf{x}\right] = \textbf{k}}{\left|  \frac{\# \left\{\textbf{p} \in A_L(\textbf{x}): \textbf{p} \leq \textbf{x} \right\} }{|\mathcal{P}|} - |A_L(\textbf{x})|  \right|}$$
is precisely the classical problem of estimating the discrepancy of random variables. The set $A_L(\textbf{x})$ fixes $|L|$ entries
of $\left[ \textbf{x} \right]$ and requires that the remaining $d-|L|$ entries do not exceed $\left[ \textbf{x} \right]$.
\begin{center}
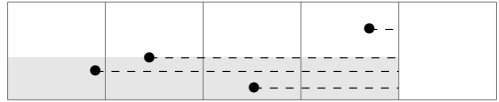
\begin{figure}[h!]
\centering
\begin{tikzpicture}[scale=1.3]
\fill[gray!20!white] (8,4) -- (12,4) -- (12,4.43)--(8,4.43);
\draw[step=1cm,gray,very thin] (8,4) grid (13,5);

\node at (8.9,4.29) {$\bullet$}; 
\node at (9.45,4.43) {$\bullet$}; 
\node at (10.52,4.12) {$\bullet$}; 
\node at (11.7,4.72) {$\bullet$}; 
\draw [thin, dashed] (8.9,4.29) -- (12,4.29);
\draw [thin, dashed] (9.45,4.43) -- (12,4.43);
\draw [thin, dashed] (10.52,4.12) -- (12,4.12);
\draw [thin, dashed] (11.7,4.72) -- (12,4.72);
\end{tikzpicture}
\caption{The worst strip is obtained from the point maximizing the star discrepancy of the projection.} \label{fig:stripe3}
\end{figure}
\end{center}
\vspace{-20pt}
By construction of jittered sampling it follows therefore that whether or not a
point $\textbf{p} \in \mathcal{P}$ is contained in $A_L(\textbf{x})$ depends only on the $|L|$ fixed coordinates.
We define a projection 
$$ \pi:  \left\{\textbf{p} \in A_L(\textbf{x}): \textbf{p} \leq \textbf{x}\right\} \rightarrow [0,N^{-\frac{1}{d}}]^{|L|}$$
by projecting onto the fixed coordinates given by $L$; see Figure \ref{fig:stripe3}. Maximizing the quantity 
$$\sup_{\left[\textbf{x}\right] = \textbf{k}}{\left|  \frac{\# \left\{\textbf{p} \in A_L(\textbf{x}): \textbf{p} \leq \textbf{x} \right\} }{|\mathcal{P}|} - |A_L(\textbf{x})|  \right|}$$
over all $\textbf{x}$ satisfying $\left[\textbf{x}\right] = \textbf{k}$ is therefore akin to determining the discrepancy of a fixed number of random points in $[0,N^{-1/d}]^{|L|}$.
The actual number of random points in question (being at most $N^{d-|L|}$) as well as the volume $[0,N^{-1/d}]^{|L|}$ both decrease as $|L|$ increase, therefore -- as we shall show -- the
case $|L| = 1$ is the most interesting one.


\subsection{An adapted Bernstein inequality.} Our proof requires the use of a Bernstein inequality; the classical Bernstein inequality assumes that a certain random variable is
compactly supported. We are dealing with a situation where the random variable is compactly supported but is with very high likelihood contained in an interval much smaller than the support. 
We therefore return to the original derivation of the Bernstein inequality and exploit that it
merely uses the existence of exponential moments -- however, using the sharp Dvoretzky-Kiefer-Wolfowitz inequality, it is possible to get upper bounds on the exponential moments.
This is contained in the following lemma.
\begin{lem} \label{expmom} Let $z_1, \dots, z_{n}$ be i.i.d. uniformly distributed random variables on $[0,1]$. For every $t > 0$
$$ \mathbb{E} \exp\left(t \sup_{0 \leq z \leq 1}{\left| \frac{\# \left\{1 \leq \ell \leq n: 0 \leq z_{\ell} \leq z\right\}}{n} -z \right|} \right) \leq 1 + \sqrt{2\pi}\frac{t}{\sqrt{n}}\exp\left(\frac{t^2}{8n}\right).$$
\end{lem}
\begin{proof} We rewrite the expectation
$$ I = \int_{0}^{\infty}{e^{t x} \left(-\frac{d}{dx} 
 \mathbb{P}\left( \sup_{0 \leq z \leq 1}{\left| \frac{\# \left\{1 \leq \ell \leq n: 0 \leq z_{\ell} \leq z\right\}}{n} -z \right|} > x \right)\right) dx}$$
 using integration by parts (which is allowed thanks to the superexponential decay guaranteed by the
Dvoretzky-Kiefer-Wolfowitz inequality)
\begin{align*}
I &= 1 + \int_{0}^{\infty}{\left(\frac{d}{dx} e^{t x} \right)  
 \mathbb{P}\left( \sup_{0 \leq z \leq 1}{\left| \frac{\# \left\{1 \leq \ell \leq n: 0 \leq z_{\ell} \leq z\right\}}{n} -z \right|} > x \right) dx}  \\
&\leq 1+  2t \int_{0}^{\infty}{e^{t x -2n x^2} dx} \leq  1+  2t e^{\frac{t^2}{8n}} \int_{0}^{\infty}{e^{-2n\left(x - \frac{t}{4n} \right)^2} dx} \leq 1 +  \sqrt{2\pi} \frac{t}{\sqrt{n}} e^{\frac{t^2}{8n}}.
\end{align*}
\end{proof}
Let $X$ now denote an arbitrary $\mathbb{R}-$valued random variable satisfying
$$  \mathbb{E} \exp\left(tX\right) \leq 1 + \sqrt{2\pi}\frac{t}{\sqrt{n}}\exp\left(\frac{t^2}{8n}\right)$$
and consider the sum of $d$ independent copies $S_d=X_1 + \dots + X_d$. To keep our manuscript self-contained, we will now repeat the usual way the Bernstein inequality
is derived from a bound on the exponential moment.
\begin{lem} \label{bernhow} We have
$$ \mathbb{P}\left(S_d \geq y\right) \leq  \left(1+\frac{\sqrt{32\pi n}}{d}y\right)^d \exp\left(-\frac{2n y^2}{d}\right).$$
\end{lem}
\begin{proof} Clearly, for any $t > 0$ we have that whenever
$$ S_d \geq y \qquad \mbox{then} \qquad e^{t(S_d - y)} \geq 1$$
and therefore
$$ \mathbb{P}\left(S_d \geq y\right)  = \mathbb{E}~1_{S_d \geq y} \leq \mathbb{E}~e^{t(S_d - y)}.$$
It remains to exploit that $S_d$ is the sum of $d$ independent copies and therefore, for every $t > 0$
\begin{align*}
\mathbb{P}\left(S_d \geq y\right) &\leq  \mathbb{E}~e^{t(S_d - y)} \leq  e^{-ty} \mathbb{E}~e^{t S_d} = e^{-t y} \prod_{j=1}^{d}{\mathbb{E} e^{t X_j}} \leq e^{-t y} \left(1 + \sqrt{2\pi}\frac{t}{\sqrt{n}}\right)^d\exp\left(\frac{t^2d}{8n}\right).
\end{align*}

Setting $t = (4ny)/d$ implies the result.
\end{proof}

\subsection{Bounds on the discrepancy of random variables.} The final ingredient is a bound on the probability of the discrepancy of $N$ independently and uniformly distributed 
random points in $[0,1]^d$
exceeding a certain limit due to Heinrich, Novak, Wasilkowski \& Wozniakowski \cite{hnww}. It states that 
$$ \mathbb{P}(D_N^* \geq 2\delta) \leq 2\left(\frac{d}{\delta} + 2\right)^d \exp\left(-\delta^2 N/2\right).$$


\section{Proof of the main statement: upper bound} 

\begin{proof} We start with a simple reduction: by definition, a complicated way of writing discrepancy is
$$ D_N^*(\mathcal{P}) = \sup_{\textbf{k} \in \left\{1, \dots, m \right\}^d}{  \sup_{\left[\textbf{x}\right] = \textbf{k}}{\left| \frac{\# \left\{\textbf{p} \in \mathcal{P}: \textbf{p} \leq \textbf{x} \right\}}{|\mathcal{P}|} - \prod_{i=1}^{d}x_i\right|} }.$$
We fix $\textbf{k} = (m,m, \dots, m)$ and carry out the argument for this special case to avoid additional notation; it is easy to see that this is the worst case for all our bounds (all of which remain true
for any other value of $\textbf{k}$; see also our proof for $d=2$ where everything is explicit). This is not surprising at all: a larger number of random points increases the likelihood for large deviations.
We use the triangle inequality and get
\begin{align*}
\sup_{\left[\textbf{x}\right] = \textbf{k}}{\left| \frac{\# \left\{\textbf{p} \in \mathcal{P}: \textbf{p} \leq \textbf{x} \right\}}{|\mathcal{P}|} - \prod_{i=1}^{d}x_i\right|} &=
 \sup_{\left[\textbf{x}\right] = \textbf{k}}{\left| \sum_{L \in \left\{0,1\right\}^d}{ \frac{\# \left\{\textbf{p} \in A_L(\textbf{x}): \textbf{p} \leq \textbf{x} \wedge \left[\textbf{p}\right] \neq \textbf{k} \right\} }{|\mathcal{P}|} - |A_L(\textbf{x})|} \right| }\\
&\leq
 \sup_{\left[\textbf{x}\right] = \textbf{k}}{\left| \sum_{L \in \left\{0,1\right\}^d \atop |L| =0}{ \frac{\# \left\{\textbf{p} \in A_L(\textbf{x}): \textbf{p} \leq \textbf{x} \wedge \left[\textbf{p}\right] \neq \textbf{k} \right\} }{|\mathcal{P}|} - |A_L(\textbf{x})|} \right| } \\
&+
 \sup_{\left[\textbf{x}\right] = \textbf{k}}{\left| \sum_{L \in \left\{0,1\right\}^d \atop |L| =1}{ \frac{\# \left\{\textbf{p} \in A_L(\textbf{x}): \textbf{p} \leq \textbf{x} \wedge \left[\textbf{p}\right] \neq \textbf{k} \right\} }{|\mathcal{P}|} - |A_L(\textbf{x})|} \right| } \\
&+
 \sup_{\left[\textbf{x}\right] = \textbf{k}}{\sum_{L \in \left\{0,1\right\}^d \atop |L| \geq 2}{\left|  \frac{\# \left\{\textbf{p} \in A_L(\textbf{x}): \textbf{p} \leq \textbf{x} \wedge \left[\textbf{p}\right] \neq \textbf{k} \right\} }{|\mathcal{P}|} - |A_L(\textbf{x})| \right|} }.
\end{align*}

The first expression, $|L| = 0$, is 0 by construction: it corresponds to a set of measure
 $$(1-N^{-\frac{1}{d}})^d \quad \mbox{while containing} \quad (N^{\frac{1}{d}} - 1)^{d} \quad \mbox{out of $N$ points.}$$

\subsection{The main contribution.} Now we analyze the case
$$  \sup_{\left[\textbf{x}\right] = \textbf{k}}{\left| \sum_{L \in \left\{0,1\right\}^d \atop |L| =1}{ \frac{\# \left\{\textbf{p} \in A_L(\textbf{x}): \textbf{p} \leq \textbf{x} \wedge \left[\textbf{p}\right] \neq \textbf{k} \right\} }{|\mathcal{P}|} - |A_L(\textbf{x})|} \right| },$$
which ultimately produce the main contribution. The sum extends over $d$ independent random variables; each single random variable corresponds precisely to the quantity estimated
in the Dvoretzky-Kiefer-Wolfowitz inequality with
$$ \#~ \mbox{of points} = (N^{\frac{1}{d}} - 1)^{d-1}.$$
First we use the projection argument to realize that we can identify our quantity of interest with a simpler probabilistic object: let $z_1, \dots, z_{n}$ be i.i.d. uniformly distributed random variables on $[0,1]$. Then
the following two random variables coincide
\begin{align*}
\sup_{\left[\textbf{x}\right] = \textbf{k}}{\left|  \frac{\# \left\{\textbf{p} \in A_L(\textbf{x}): \textbf{p} \leq \textbf{x} \wedge \left[\textbf{p}\right] \neq \textbf{k} \right\} }{|\mathcal{P}|} - |A_L(\textbf{x})|  \right|}
=
 \frac{1}{N^{\frac{1}{d}}}\sup_{0 \leq z \leq 1}{\left| \frac{\# \left\{1 \leq \ell \leq (N^{\frac{1}{d}}-1)^{d-1} : 0 \leq z_{\ell} \leq z\right\}}{n} -z \right|}  \end{align*}
Plugging this into Lemma \ref{expmom}, we get
$$ \mathbb{E} \exp\left(t \sup_{\left[\textbf{x}\right] = \textbf{k}}{\left|  \frac{\# \left\{\textbf{p} \in A_L(\textbf{x}): \textbf{p} \leq \textbf{x} \wedge \left[\textbf{p}\right] \neq \textbf{k} \right\} }{|\mathcal{P}|} - |A_L(\textbf{x})|  \right|} \right) \leq 1 + \sqrt{2\pi}\frac{tN^{-\frac{1}{d}}}{ (N^{\frac{1}{d}}-1)^{\frac{d-1}{2}}  } \exp\left( \frac{t^2 N^{-\frac{2}{d}}}{8 (N^{\frac{1}{d}}-1)^{d-1}   }\right).$$
The next few steps requires algebraic manipulation of that expression. This can be done in all the standard ways but we consider it instructive to keep the focus on the precise exponents. We shall
therefore introduce $\varepsilon_0 > 0$ and assume $N$ to be so big that
$$  (N^{\frac{1}{d}}-1)^{d-1}  \geq (1-\varepsilon_0)N^{\frac{d-1}{d}}.$$
In the end we will see that a small enough $\varepsilon_0>0$ (and thus a sufficiently large $N$) will suffice. Trivially,
\begin{align*}   \mathbb{E} \exp\left(t \sup_{\left[\textbf{x}\right] = \textbf{k}}{\left|  \frac{\# \left\{\textbf{p} \in A_L(\textbf{x}): \textbf{p} \leq \textbf{x} \wedge \left[\textbf{p}\right] \neq \textbf{k} \right\} }{|\mathcal{P}|} - |A_L(\textbf{x})|  \right|} \right) 
&\leq 1 + \sqrt{2\pi}\frac{tN^{-\frac{1}{d}}}{ (N^{\frac{1}{d}}-1)^{\frac{d-1}{2}}  } \exp\left( \frac{t^2 N^{-\frac{2}{d}}}{8 (N^{\frac{1}{d}}-1)^{d-1}   }\right) \\
&\leq 1 +  \frac{\sqrt{2\pi}}{\sqrt{1-\varepsilon_0}}\frac{t}{ N^{\frac{d+1}{2d}} } \exp\left( \frac{t^2 }{8(1-\varepsilon_0) N^{\frac{d+1}{d}}   }\right) \\
&\leq \left( 1 +  \frac{\sqrt{2\pi}}{\sqrt{1-\varepsilon_0}}\frac{t}{ N^{\frac{d+1}{2d}} } \right)\exp\left( \frac{t^2 }{8(1-\varepsilon_0)  N^{\frac{d+1}{d}}   }\right) 
\end{align*}
Next, we consider the sum of $d$ i.i.d. random variables 
$$ S_d =  \sum_{L \in \left\{0,1\right\}^d \atop |L| =1}{ \sup_{\left[\textbf{x}\right] = \textbf{k}} \left| \frac{\# \left\{\textbf{p} \in A_L(\textbf{x}): \textbf{p} \leq \textbf{x} \wedge \left[\textbf{p}\right] \neq \textbf{k} \right\} }{|\mathcal{P}|} - |A_L(\textbf{x})|\right|},$$
where each one satisfies the bound on the exponential moment outlined above. Using the standard derivation of Bernstein inequalities directly as in the proof of Lemma \ref{bernhow}, we obtain by setting the free
parameter in the argument to be
$$ t = \frac{4(1-\varepsilon_0)y N^{\frac{d+1}{d}}}{d}$$
the inequality
$$ \mathbb{P}(S_d \geq y) \leq \left(1 + \sqrt{2\pi}\sqrt{1-\varepsilon_0}  \frac{4N^{\frac{1}{2} + \frac{1}{2d}}y}{d}\right)^d\exp \left(-2(1-\varepsilon_0) \frac{N^{1+\frac{1}{d}} y^2}{d}\right).$$
In particular, this implies
$$ \mathbb{P}\left(S_d \geq \frac{C\sqrt{d \log{N}}}{N^{\frac{1}{2} + \frac{1}{2d}}}\right) \leq  \left( 1 + \sqrt{2\pi}4\sqrt{1-\varepsilon_0} C \sqrt{\frac{\log{N}}{d}} \right)^d
\exp \left(-2(1-\varepsilon_0) C^2 \log{N}\right).$$
However, since we are interested in bounding the expectation of the discrepancy, we need to somehow incorporate cases
where one of the cubes violates that condition: in that case we simply assume that the discrepancy assumes maximal value 1. Altogether,

$$ \mathbb{E} D_N^*(\mathcal{P}) \leq \frac{C\sqrt{d \log{N}}}{N^{\frac{1}{2} + \frac{1}{2d}}} + N  \mathbb{P}\left(S_d \geq \frac{C\sqrt{d \log{N}}}{N^{\frac{1}{2} + \frac{1}{2d}}}\right) .$$
It remains to see whether we can make the second term smaller than the first one. A simple computation shows that this requires
$$ 2(1-\varepsilon_0) C^2 \geq \frac{3}{2} + \frac{1}{2d} \qquad \mbox{or, equivalently,} \qquad C > \sqrt{\frac{3}{4} + \frac{1}{4d}}$$
as $\varepsilon_0$ can be made arbitrarily small.
This shows that the way we formulated Theorem \ref{main} throws away a small gain. We could actually show that
$$\mathbb{E} D_N^*(\mathcal{P})  \leq  \left(\sqrt{\frac{3}{4} + \frac{1}{4d}} + \varepsilon\right) \frac{\sqrt{d \log{N} }}{N^{\frac{1}{2} + \frac{1}{2d}}}$$
for every $\varepsilon > 0$ and $N$ sufficiently large depending on $\varepsilon$. However, we chose to omit that formulation for the sake of brevity and
because it yields merely an improvement of the constant.

\subsection{The remaining contributions.}
It remains to study the case $|L| \geq 2$ and show that the additional contribution is small with high likelihood. Let now $2 \leq \ell \leq d$ be arbitrary and consider
$$  \sum_{L \in \left\{0,1\right\}^d \atop |L| = \ell}{\sup_{\left[\textbf{x}\right] = \textbf{k}}{ \left|  \frac{\# \left\{\textbf{p} \in A_L(\textbf{x}): \textbf{p} \leq \textbf{x} \wedge \left[\textbf{p}\right] \neq \textbf{k} \right\} }{|\mathcal{P}|} - |A_L(\textbf{x})| \right|} }$$
Note that interchanging the sum with the supremum, as we did, strictly increases the quantity. In the case of $|L| = 1$, the geometric structure of the decomposition ensured that interchanging
those quantities has no effect; here, these quantities are actually intertwined in a complicated way and we lose in the process of interchanging; however, since this is not the main term, the loss
is acceptable. The sum contains $\binom{d}{\ell}$ terms and we treat all of them independently.
Every single set contains at most 
$$(N^{\frac{1}{d}} - 1)^{d-\ell} \quad\mbox{points}.$$
As before, we can now project every set
$$ \left\{\textbf{p} \in A_L(\textbf{x}): \textbf{p} \leq \textbf{x} \wedge \left[\textbf{p}\right] \neq \textbf{k}\right\} \rightarrow [0,N^{-\frac{1}{d}}]^{\ell}.$$
By employing the full strength of the inequality of Heinrich, Novak, Wasilkowski \& Wozniakowski, we get for $N$ sufficiently large
\begin{align*}
\mathbb{P} \left( \sup_{\left[\textbf{x}\right] = \textbf{k}}{ \left|  \frac{\# \left\{\textbf{p} \in A_L(\textbf{x}): \textbf{p} \leq \textbf{x} \wedge \left[\textbf{p}\right] \neq \textbf{k} \right\} }{|\mathcal{P}|} - |A_L(\textbf{x})| \right|} > \frac{1}{\binom{d}{\ell}d} \right. & \left. \frac{1}{N^{\frac{1}{2} + \frac{1}{2d}}}\right) \\
&\leq (2 + o(1))\exp\left(-c_{d,\ell} N^{\frac{\ell-1}{d}}\right)
\end{align*}
for some constant $c_{d,\ell} > 0$ that could be explicitly computed.
Since $\ell \geq 2$, this probability decays faster than any polynomial in $N$. Hence, for $N$ sufficiently large, the
probability of this event for any $\ell$ and any of the $N$ cubes goes to 0. Therefore, we can bound
$$  \sum_{\ell} \sup_{\left[\textbf{x}\right] = \textbf{k}}{\sum_{L \in \left\{0,1\right\}^d \atop |L| = \ell}{\left|  \frac{\# \left\{\textbf{p} \in A_L(\textbf{x}): \textbf{p} \leq \textbf{x} \wedge \left[\textbf{p}\right] \neq \textbf{k} \right\} }{|\mathcal{P}|} - |A_L(\textbf{x})| \right|} } \leq \frac{1}{N^{\frac{1}{2} + \frac{1}{2d}}}$$
for all $N$ cubes with a likelihood converging to 1 faster than $1-N^{-c}$ for every fixed $c > 0$. This implies the result.
\end{proof}


\section{Proof of the Statement: lower bound}
Our derivation of the lower bound comes from considering again only the cube with coordinates $(m,m,\dots,m)$. We have a relatively good understanding of the underlying processes: as $N$ becomes large, the main contribution comes from the $d$ slices containing a large proportion of the points whereas
all other $2^d-d-1$ slices  have strictly smaller proportion of points that decreases as their codimension increases. Furthermore, the $d$ major
slices are independent. This motivates the structure of our argument: we compute the average discrepancy contribution of one slice in isolation, show that we can pick a point
attaining that lower bound simultaneously for all $d$ slices, show that at least $d/2$ have the same sign (which is trivial) and then show that with very high likelihood the quantities contributed by the remaining sets in the partition
are an entire order of magnitude smaller.

\begin{proof} We start by analyzing one major slice. 
Setting 
$$ X_n =\sup_{0 \leq z \leq 1}{\left| \frac{\# \left\{1 \leq \ell \leq n: 0 \leq z_{\ell} \leq z\right\}}{n} -z \right|}$$
for the one-dimensional discrepancy, we recall that we have
$$ \left|  \frac{\# \left\{\textbf{p} \in A_L(\textbf{x}): \textbf{p} \leq \textbf{x} \wedge \left[\textbf{p}\right] \neq \textbf{k} \right\} }{|\mathcal{P}|} - |A_L(\textbf{x})|  \right| = 
\frac{X_n}{ N^{\frac{1}{d}}} $$
for $|L| = 1$ and $n=(N^{1/d}-1)^{d-1}$.

The quantity $X_n$, however, is rather well understood: it is actually known to converge in distribution to the Kolmogorov distribution and thus
\begin{align*}
 \lim_{n \rightarrow \infty}{\sqrt{n} ~\mathbb{E} X_n}  &= \int_{0}^{\infty}{x \frac{d}{dx} \left(-2\sum_{k=1}^{\infty}{(-1)^{k-1}e^{-2 k^2 x^2}}\right) dx} \\
&= \sum_{k=1}^{\infty}{(-1)^{k-1} \int_{\mathbb{R}}{e^{-2k^2 x^2}dx}} \\
&= \sqrt{\frac{\pi}{2}} \sum_{k=1}^{\infty}{ \frac{(-1)^{k-1}}{k}}\\
&= \sqrt{\frac{\pi}{2}} \log{2} \sim 0.86873\dots
\end{align*}
Note, furthermore, that
$ (N^{\frac{1}{d}} - 1)^{d-1} \leq N^{\frac{d-1}{d}}$.
Therefore, the expected contribution of a single slice is
$$ \mathbb{E} \frac{X_{(N^{\frac{1}{d}}-1)^{d-1}}}{N^{\frac{1}{d}}} \geq \frac{c}{N^{\frac{d-1}{2d}}}\frac{1}{N^{\frac{1}{d}}} = \frac{c}{N^{\frac{1}{2} + \frac{1}{2d}}},$$
for every $c < \sqrt{\pi/2}\log{2}$ and every $N$ sufficiently large depending on $c$. Note that this is for the \textit{absolute value} of one
discrepancy function. However, we do not know which sign will actually arise (whether there are too many or too few points). Among $d$ slices we expect $d/2$ slices
to have too many points and $d/2$ to have too few points; since we can actually set entire slices to zero by setting the coordinate of the corresponding slice to be $(N^{\frac{1}{d}}-1)/N$ (the lower
corner of the cube), we can therefore guarantee a total expectation of
$$ \mathbb{E} D_N^*(\mathcal{P}) \geq \frac{d}{2} \frac{c}{N^{\frac{d+1}{2d}}}.$$
\end{proof}

\section{Remarks and Comments}

\subsection{A change of regime.} First, we explain why it is natural to expect a change of behavior in the discrepancy of fully random points and the discrepancy of jittered sampling at $N \sim d^d$. The first heuristic
is as follows: the construction allows us to discard all points in the 'big box' and merely consider the slices instead. The fraction of points in the big box is
$$ \frac{(N^{\frac{1}{d}} - 1)^d}{N} = \left(1-\frac{1}{N^{\frac{1}{d}}}\right)^d \sim \begin{cases} 1 \qquad &\mbox{if}~N \gg d^d \\
0 \qquad &\mbox{if}~N \ll d^d\end{cases}.$$
Note that for $N = d^d$, the fraction converges to $1/e$. Put differently, for $N = d^d$, there is a constant fraction of the points in the big box but we still have to consider a constant fraction
of $\sim 1-e^{-1}$ points.
Let us now refine the heuristic somewhat: in our case, the primary contribution comes from $d$ random variables of size $\sim 1/N^{1/2 + 1/(2d)}$ whereas $N$ independently and identically
distributed random variables give rise to a discrepancy of $\sqrt{d/N}$. Altogether, we have to compare
$$ \sqrt{\frac{d}{N}} \qquad \mbox{with} \qquad \frac{d}{N^{\frac{1}{2} + \frac{1}{2d}}} =  \sqrt{\frac{d}{N}} \frac{\sqrt{d}}{N^{\frac{1}{2d}}},$$
where the last factor is $= 1$ for $N = d^d$.

\subsection{Explicit point sets in the large regime.} Assuming our upper bound to hold unconditionally, we improve on fully random points as soon as $N \gtrsim d^d$. The
purpose of this section is a short heuristic suggesting that deterministic constructions with the same number $N$ of points are not (yet provably) better than point sets obtained from jittered sampling. 
We compare our results with the leading term in the best known bound on the discrepancy of Hammersley point sets derived from recent results of Atanassov \cite{ata}. 
Following \cite[Theorem 3.46]{dickpill}, the star discrepancy of the Hammersley point set $\mathcal{H}_N$ consisting of $N$ points in $[0,1]^d$ is bounded by
\begin{equation*}
D_N^*(\mathcal{H}_N) < \frac{7}{2^{d-1}(d-1)} \frac{(\log N)^{d-1}}{N} + \mathcal{O} \left( \frac{(\log N)^{d-2}}{N} \right),
\end{equation*}
if $\mathcal{H}_N$ is generated from the first $d-1$ prime numbers. (Replacing the first $d-1$ prime numbers by a set of larger coprime integers increases the constant.)
Altogether, for $N = (2d)^{2d}$ points, we thus get a bound of roughly
$$ \sim \frac{1}{ 2^{d-1}} \frac{(\log(2d))^{d-1}}{(2d)^{d+1}} \qquad \mbox{whereas our bound yields} \qquad \frac{d \sqrt{2} \sqrt{\log{2d}}}{(2d)^{d+1} } .$$
This suggests that the known discrepancy bound for the Hammersley point set is a superexponential factor $1/d \,  (\log{(2d)}/2)^{d-3/2}$ bigger than our bound.

\subsection{Some numerical results.} The following conjecture seems exceedingly natural: the expected star-discrepancy of $N = m^d$ independently and uniformly 
distributed points in $[0,1]^d$ should always be bigger than the expected star-discrepancy of $N = m^d$ points obtained
from jittered sampling. This conjecture is supported by the results of our numerical experiments shown in Table 1. For the computation of the star discrepancy we used a recent implementation of the Dobkin-Eppstein-Mitchell algorithm \cite{dob} by Magnus Wahlstr\"{o}m, which is freely available online \cite{wahl} and which computes the star discrepancy exactly; for details on the implementation we refer to \cite{doerr}.

\begin{table}[h!] 
\begin{tabular}{ | c | c c c | c c | c c| }
\hline
 					&  			& $d=2$ 		& 							& $d=3$ &  & $d=5$&\\
\hline
					&  $m=5$ 		& $m=10$ 		& $m=20$ 		& $m=5$& $m=10$ & $m=3$& $m=5$\\
\hline 
Jittered Sampling	&  	$0.1518$	&  	$0.0629$	& 	$0.0243$	&	$0.0932$	&	$0.0279$	& $0.1046$&  $ 0.0259$\\
Random Sets		&  	$0.2180$	&  	$0.1232$	& 	$0.0624$	& 	$0.1318$	&	$0.0542$	& $0.1200$& $0.0331$\\
\hline
\end{tabular} 
\\[6pt]
\caption{Mean discrepancy of 10 experiments with $N=m^d$ points in $[0,1]^d$.}
\end{table}

\subsection{The lower bound should not be optimal.} The main heuristic for why the lower bound should not be optimal
lies in its derivation: we computed the expected discrepancy that we could expect by restricting
ourselves to rectangles $[0, \textbf{x}]$ with $x$ being in the cube indexed by $(m,m,\dots,m)$. However, we could
very well repeat the same computation with another cube on the diagonal, say, the cube $(m-1, m-1, \dots, m-1)$.
The construction of jittered sampling implies that the discrepancy that can be attained for each single one of these
cubes is an independent random variable. Moreover, the random variable computed for a single cube had an
expectation of
$$\mathbb{E} (\mbox{contribution of a single cube}) \sim \frac{d}{N^{\frac{1}{2} + \frac{1}{2d}}}.$$
\begin{center}
\begin{figure}[h!]
\centering
\begin{tikzpicture}[scale=1]
\fill[gray!20!white] (12.67,0) -- (12.67,4) -- (12,4)--(12,0);
\fill[gray!20!white] (8,4) -- (12,4) -- (12,4.19)--(8,4.19);
\draw[dashed,thick] (12.67,0) -- (12.67,4.19) -- (8,4.19);

\fill[gray!20!white] (11.3,0) -- (11.3,3) -- (11,3)--(11,0);
\fill[gray!20!white] (8,3.5) -- (11,3.5) -- (11,3)--(8,3);
\draw[dashed,thick] (11.3,0) -- (11.3,3.5) -- (8,3.5);

\fill[gray!20!white] (10.3,0) -- (10.3,2) -- (10,2)--(10,0);
\fill[gray!20!white] (8,2.5) -- (10,2.5) -- (10,2)--(8,2);
\draw[dashed,thick] (10.3,0) -- (10.3,2.5) -- (8,2.5);

\draw[step=1cm,gray,very thin] (8,0) grid (13,5);

\node at (8.31,0.23) {$\bullet$}; 
\node at (8.7,1.5) {$\bullet$}; 
\node at (8.2,2.7) {$\bullet$}; 
\node at (8.3,3.34) {$\bullet$}; 
\node at (8.9,4.29) {$\bullet$}; 
\node at (9.3,0.74) {$\bullet$}; 
\node at (9.55,1.56) {$\bullet$}; 
\node at (9.1,2.83) {$\bullet$}; 
\node at (9.65,3.42) {$\bullet$}; 
\node at (9.45,4.43) {$\bullet$}; 
\node at (10.34,0.41) {$\bullet$}; 
\node at (10.67,1.66) {$\bullet$}; 
\node at (10.52,2.42) {$\bullet$}; 
\node at (10.67,3.37) {$\bullet$}; 
\node at (10.52,4.12) {$\bullet$}; 
\node at (11.5,0.82) {$\bullet$}; 
\node at (11.7,1.58) {$\bullet$}; 
\node at (11.3,2.21) {$\bullet$}; 
\node at (11.9,3.14) {$\bullet$}; 
\node at (11.7,4.72) {$\bullet$}; 
\node at (12.23,0.63) {$\bullet$}; 
\node at (12.56,1.61) {$\bullet$}; 
\node at (12.89,2.25) {$\bullet$}; 
\node at (12.17,3.57) {$\bullet$}; 
\node at (12.43,4.89) {$\bullet$}; 
\end{tikzpicture}
\caption{Computing discrepancy w.r.t. many different cubes on the diagonal.}
\end{figure}
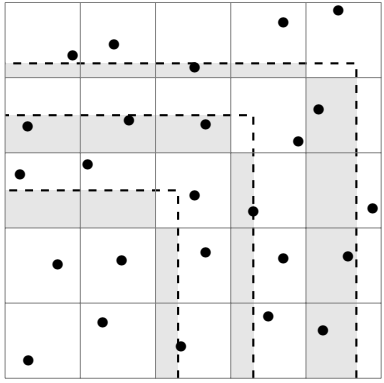
\end{center}
We recall that this discrepancy arose (in leading order) from the contribution of $d$ slices and the Dvoretzky-Kiefer-Wolfowitz inequality
suggests that each slice contributes a random variable that decays essentially like a Gaussian. This suggests the following heuristic
for the random variable
$$ \mbox{contribution of a single cube} \sim \frac{d+\mathcal{N}(0,\sqrt{d})}{N^{\frac{1}{2} + \frac{1}{2d}}} = \frac{d+\sqrt{d}\mathcal{N}(0,1)}{N^{\frac{1}{2} + \frac{1}{2d}}}$$
We recall the Fisher-Tippett-Gnedenko theorem (see e.g. \cite{gned}) which implies that for
$X_1, X_2, \dots, X_n$ independent $\mathcal{N}(0,1)-$distributed random variables, we have that
$$ \mathbb{E} \max(X_1, \dots, X_n) \sim \sqrt{\log{n}},$$
where $\sim$ hides some absolute constants. There are $N^{\frac{1}{d}}$ cubes on the diagonal, which suggests
$$ \sup_{\mbox{cube}~Q~\mbox{on the diagonal}} \mathbb{E} (\mbox{contribution of} ~Q) \sim \frac{d + \sqrt{d}\sqrt{\log{N^{\frac{1}{d}}}}}{N^{\frac{1}{2} + \frac{1}{2d}}} = \frac{d + (\log{N})^{\frac{1}{2}}}{N^{\frac{1}{2} + \frac{1}{2d}}}.$$
This heuristic might actually be very close to the truth. Note that, as we move down the diagonal, the number of random points influencing the discrepancy is actually decreasing
and large deviations become increasingly unlikely. This decrease is slight if $N$ is big but not so slight for small $N$, which is why all of this should be a gross overestimation
for $N \lesssim d^d$. The difficulty
in making this reasoning precise is that one seems to require an inverse Dvoretzky-Kiefer-Wolfowitz inequality, which guarantees that the likelihood
of single slices contributing large values is comparable with the Gaussian bound from above; it seems possible that a suitable application of the Komlos-Major-Tusnady approximation \cite{kmt, kmt1} could be helpful in making further progress in that direction.

\section{Proof of the Theorem 1.2.}
\subsection{Main statement.} 
\begin{proof}
 We fix $N \in \mathbb{N}$ and a partition
$$ [0,1]^d = \bigcup_{i=1}^{N}{\Omega_i} \qquad \mbox{such that} \quad \forall~1 \leq i \leq N: |\Omega_i| = \frac{1}{N}$$
throughout the proof. To avoid unnecessary formalism, we will use $\mathcal{P}_{\Omega}$ for a jittered set generated w.r.t. $\Omega$ and use $\mathcal{P}_N$ to denote a set of $N$ fully random points and we denote integration in the respective probability spaces by
 $\int_{\mathcal{P}_{\Omega}}$ and $\int_{\mathcal{P}_{N}}$, respectively. The proof proceeds by interchanging the order of integration (this was previously useful in a very similar context in \cite{sst}) and applying a convexity argument. We start with
\begin{align*}
\mathbb{E} \mathcal{L}^2_2(\mathcal{P}_{\Omega}) &=  \int_{\mathcal{P}_{\Omega}} \int_{[0,1]^d}{ \left| \frac{ \#\mathcal{P}_{\Omega} \cap [0, \textbf{x}]}{N} -  \left|  [0,\textbf{x}]\right|   \right|^2 d\textbf{x} d\omega }\\
&=   \int_{[0,1]^d} \int_{\mathcal{P}_{\Omega}}{ \left| \frac{ \#\mathcal{P}_{\Omega} \cap [0, \textbf{x}]}{N} -  \left|  [0,\textbf{x}]\right|   \right|^2  d\omega d\textbf{x}}.
\end{align*}
We will now fix $\textbf{x}$. By construction, we have
$$ \mathbb{E}  \frac{ \#\mathcal{P}_{\Omega} \cap [0, \textbf{x}]}{N} = \int_{\mathcal{P}_{\Omega}}{ \frac{ \#\mathcal{P}_{\Omega} \cap [0, \textbf{x}]}{N} d\omega} = \frac{1}{N}\sum_{i=1}^{N}{\frac{\left| \Omega_i \cap [0,\textbf{x}]\right|}{|\Omega_i|} } =  \sum_{i=1}^{N}{\left| \Omega_i \cap [0,\textbf{x}] \right|}  = \left|  [0,\textbf{x}] \right|.$$
Therefore, we have that 
$$  \int_{\mathcal{P}_{\Omega}}{ \left| \frac{ \#\mathcal{P}_{\Omega} \cap [0, \textbf{x}]}{N} -  \left|  [0,\textbf{x}]\right|   \right|^2 d\omega} = \mbox{var}\left(  \frac{ \#\mathcal{P}_{\Omega} \cap [0, \textbf{x}]}{N} \right) = \frac{1}{N^2}  \mbox{var}\left(   \#\mathcal{P}_{\Omega} \cap [0, \textbf{x}] \right) .$$
However, there is another way to understand that random variable: using $\mathcal{B}(n,p)$ to denote a Bernoulli random variable with success probability $p$, we have, by construction, a sum of independent (but \textit{not} identical) random variables
$$   \#\mathcal{P}_{\Omega} \cap [0, \textbf{x}]= \mathcal{B}\left(1,\frac{|\Omega_1 \cap [0,\textbf{x}]|}{|\Omega_1|}\right) +  \mathcal{B}\left(1,\frac{|\Omega_2 \cap [0,\textbf{x}]|}{|\Omega_2|}\right) + \dots +  \mathcal{B}\left(1,\frac{|\Omega_N \cap [0,\textbf{x}]|}{|\Omega_N|}\right) .$$
If $X,Y$ are independent random variables, then $\mbox{var}(X+Y) = \mbox{var}(X) + \mbox{var}(Y)$ and thus, using $|\Omega_i| = 1/N$
\begin{align*}
 \mbox{var} \left( \#\mathcal{P}_{\Omega} \cap [0, \textbf{x}] \right) &= \sum_{i=1}^{N}{ \frac{|\Omega_i \cap [0,\textbf{x}]|}{|\Omega_i|}   \left(1 - \frac{|\Omega_i \cap [0,\textbf{x}]|}{|\Omega_i|} \right)  }\\
 &= N |[0,\textbf{x}]| -  N^2\sum_{i=1}^{N}{  \left( |\Omega_i \cap [0,\textbf{x}]|  \right)^2 }. 
 \end{align*}
Cauchy-Schwarz implies that
$$  |[0,\textbf{x}]| ^2 = \left(\sum_{i=1}^{N}{  \left( |\Omega_i \cap [0,\textbf{x}]|  \right) }\right)^2 \leq N\sum_{i=1}^{N}{  \left( |\Omega_i \cap [0,\textbf{x}]|  \right)^2 }.$$
This implies
\begin{align*}
\mbox{var} \left( \#\mathcal{P}_{\Omega} \cap [0, \textbf{x}] \right)  &= N |[0,\textbf{x}]| -  N^2\sum_{i=1}^{N}{  \left( |\Omega_i \cap [0,\textbf{x}]|  \right)^2 } \\
&\leq N |[0,\textbf{x}]| -  N|[0,\textbf{x}]|^2.
\end{align*}
It remains to compute  $\mbox{var} \left( \#\mathcal{P}_{N} \cap [0, \textbf{x}] \right)$, which is easy. Let $X$ be a Bernoulli random variable with
$$ \mathbb{P}(X = 1) = |[0, \textbf{x}]| \qquad \mbox{and} \qquad  \mathbb{P}(X = 1) =1 -  |[0, \textbf{x}]|$$
and let $X_1, \dots, X_N$ be $N$ independent copies. Then 
\begin{align*}\mbox{var}\left( \#\mathcal{P}_{N} \cap [0, \textbf{x}] \right) &= \mbox{var}\left( X_1 + \dots + X_N \right) \\
&=  \mbox{var}\left( X_1 \right) + \dots +  \mbox{var}\left( X_N \right) \\
&= N  \mbox{var}\left( X \right)\\
&= N  |[0, \textbf{x}]| \left(1 -  |[0, \textbf{x}]| \right).
\end{align*}
Therefore, for every $\textbf{x} \in [0,1]^d$,
$$\mbox{var} \left( \#\mathcal{P}_{\Omega} \cap [0, \textbf{x}] \right) \leq \mbox{var} \left( \#\mathcal{P}_{N} \cap [0, \textbf{x}] \right)$$
and since $\textbf{x}$ was arbitrary, we get the desired result by integration over $[0,1]^d$.
\end{proof}


\end{document}